\documentclass[11pt]{article}

\usepackage[alphabetic]{amsrefs}
\usepackage{amsmath,amssymb,amsthm,enumerate,tikz,pxfonts}
\usepackage[all,cmtip]{xy}
\usepackage[applemac]{inputenc}

\def \1{\mathds{1}}
\def \al{\alpha}
\def \bs{\backslash}
\def \C{{\mathbb C}}

\def \CM{\mathcal M}

\def \CO{{\cal O}}
\def \CP{{\cal P}}

\def \CT{{\cal T}}

\def \diag{\operatorname{diag}}

\def \disc{\operatorname{disc}}

\def \eps{\varepsilon}

\def \Ga{\Gamma}

\def \ga{\gamma}
\def \H{{\mathbb H}}

\def \LG{\operatorname{LG}}
\def \M{\operatorname M}

\def \N{{\mathbb N}}
\def \ol{\overline}

\def \ph{\varphi}
\def \PSL{\operatorname{PSL}}
\def \Q{{\mathbb Q}}
\def \R{{\mathbb R}}
\def \Re{\operatorname{Re}}
\def \red{\mathrm{red}}
\def \SL{\operatorname{SL}}

\def \tr{\operatorname{tr}}

\def \Z{{\mathbb Z}}
\def \({\left(}
\def \){\right)}

\newcommand{\e}
[1]{\emph{#1}\index{#1}}

\newcommand{\mat}
[4]{\(\begin{matrix}#1 & #2 \\ #3 & #4\end{matrix}\)}

\newcommand{\smat}
[4]{\(\begin{smallmatrix}#1 & #2 \\ #3 & #4\end{smallmatrix}\)}

\renewcommand{\sp}
[1]{\left\langle #1\right\rangle}

\newcommand{\stack}
[2]{\genfrac{}{}{0pt}{1}{#1}{#2}}

\newtheorem{theorem}{Theorem}[section]

\newtheorem{question}[theorem]{Question}

\theoremstyle{definition}
\newtheorem{definition}[theorem]{Definition}
\newtheorem{example}[theorem]{Example}
\newtheorem{remark}[theorem]{Remark}

\begin{document}

\pagestyle{myheadings} \markright{Closed geodesics and bounded gaps}

\title{Closed geodesics and bounded gaps}
\author{Anton Deitmar\thanks{This paper was written while the author was invited to the DFG-Sonderforschungsbereich 878 "Gruppen, Geometrie und Aktionen" in M\"unster. He wishes to express his gratitude to all people there for their warm hospitality.}}
\date{}
\maketitle

{\bf Abstract:} The bounded gaps property of the prime numbers, as proven by Yitang Zhang, is considered for sequences of lengths of closed geodesics, which by the theory of Selberg zeta functions are the geometric analogue of the prime numbers.
It turns out that the property holds for congruence subgroups and is false for a dense set in Teichmüller space.

$$ $$

\tableofcontents

\section*{Introduction}

In the paper \cite{Zhang}, Yitang Zhang proved that there are infinitely many primes $p,q$ such that $|p-q|<C$, where $C>0$ is an explicit constant.
In \cite{Maynard}, James Maynard improved the bound to $C=600$. The current state of the Polymath8 project in improving the constant is $C=246$.
Viewing Ruelle zeta functions as a geometric analogue of the Riemann zeta functions, one is led to ask whether a similar result holds for the latter.

For a sequence $(x_n)$ of real numbers $x_n>0$ with $x_{n+1}>x_n$ we define the \e{limit gap} of the sequence as
$$
\LG((x_n))= \liminf_{n}|x_{n+1}-x_n|.
$$
Zhang's result says that for the sequence $\pi$ of prime numbers one has $\LG(\pi)<7\times 10^7$, where before it wasn't even known to be finite.
By Maynard's result we know $\LG(\pi)\le 600$ and, as one finds on their weibsite, the Polymath8 project is down to $\LG(\pi)\le 246$.
The number $\LG(\pi)$ is an integer $\ge 2$ and 
the twin prime conjecture says that $\LG(\pi)$ should be equal to $2$.

We say that an increasing sequence $x=(x_n)$ has  \e{bounded gaps}, if $0<\LG(x)<\infty$.
Zhang has proved that the ascending enumeration of the prime numbers has bounded gaps.
Geometric zeta functions of Ruelle, Selberg and Ihara are defined by Euler products which don't run over the primes but over lengths of closed geodesics in locally symmetric spaces or finite graphs respectively.
The sequences of lengths, or rather, the exponentials of the lengths will be called the \e{length spectra} of the underlying spaces.
So the length spectra may be viewed as a geometric analogue of the primes.
This paper is concerned with the question which length spectra  have bounded gaps.
It is shown that congruence subgroups of $\SL_2(\Z)$ define spaces with the bounded gaps property.
On the other hand, there is a dense set in Teichmüller space where the bounded gaps property does not hold.
It remains open whether general arithmetic groups do have  bounded gaps.
They would, if the property was hereditary to subgroups of finite index, which is not known.

As an analogue of the Teichmüller space we define a moduli space of metric graphs in which case limit gaps can be computed rather explicitly.

The author thanks the users named 'quid' and 'Impossible Hippopotamus' of mathoverflow.net for their helpful remarks.

\section{Metric graphs}
Recall that for a sequence $(x_n)$ of real numbers $x_n>0$ with $x_{n+1}>x_n$ we define the \e{limit gap} of the sequence as
$$
\LG((x_n))= \liminf_{n}|x_{n+1}-x_n|.
$$

\begin{remark}\label{rem1.1}
Note that $\LG((x_n))>0$ implies that 
$$
\#\big\{n\in\N:x_n\le T\big\}=O(T)\qquad \text{as }T\to\infty.
$$
\end{remark}

\begin{example}\label{Ex1.1}
For any real $a>1$ the limit gap of the sequence $a^n$ is $+\infty$.
\begin{proof}
For every $n\in\N$ one has
$$
|a^{n+1}-a^n|=a^n(a-1),
$$
so the consecutive gaps $|a^{n+1}-a^n|$ increase with $n$ and grow over all bounds.
\end{proof}
\end{example}

Metric graphs serve as the discrete model for more general geodesic spaces.
In this case, one has rather good control over the length spectra, which can be computed rather explicitly.

By a \e{graph} we mean a pair $(V,\ph)$ where $V$ is a countable set, called the \e{vertex set} and $\ph:E\to \CP_{1,2}(V)$ is a map, where $E$ is a set, the \e{edge set}, and $\CP_{1,2}(V)$ is the set of subsets of $V$ which contain one or two elements.
So loops and multiple edges are generally allowed.
A \e{geometric realization} of a graph $(V,E)$ is obtained a follows. For each edge $k$ fix a copy $I_k$ of the unit interval and choose a marking of the endpoints with $a,b$ if $k=\{a,b\}$. In the topological space $\coprod_{k\in E}I_k$, which is the disjoint union of all these intervals, we identify all endpoints marked with the same element of $V$.
The resulting space is a geometric realization of the graph.
A \e{metric graph} is obtained from this by  attaching to each edge $e$ a length $l(e)\in (0,\infty)$.
One then derives an internal metric on the geometric realization, if the graph is connected. One defines the distance of points which lie on different connected components to be $+\infty$ and thus gets a generalized metric on the entire graph.
The \e{moduli space} $\CM(V,\ph)$ is defined to be the set of all metrics in this sense, i.e., the set of all maps $l:E\to (0,\infty)$.
It carries a natural compact-open topology generated by all sets of the form
$$
L(K,U)=\{ l\in\CM(V,\ph): l(K)\subset U\},
$$
where $U\subset (0,\infty)$ is an open subset and $K\subset E$ is finite.

If $E$ is finite itself, the moduli space can be naturally compactified, the boundary consisting of moduli spaces of degenerations of $(V,E)$.

For a finite metric graph $X$ we define the \e{Ihara zeta function} \cites{Ihara,Serre,Sunada,Hashimoto,Bass} to be
$$
Z_X(s)=\prod_{c}(1-N(c)^{-s})^{-1},
$$
where the product runs over all closed geodesics, $N(c)=e^{l(c)}$ and $l(c)$ is the length of the geodesic $c$.

\begin{theorem}
For a finite metric graph $X$, the product defining the Ihara zeta function converges for $\Re(s)$ large enough and the resulting function can be extended to a meromorphic function on $\C$ which is the inverse of a polynomial in the entries $e^{-sl(k)}$, where $k$ runs through the finite set of edges.

If a sequence $(X_j)$ of finite metric graphs converges in the compactified moduli space $\CM(V,\ph)$ to a graph $X$, then the inverses of the corresponding Ihara zeta functions $Z_{X_j}(s)^{-1}$ converge locally uniformly on $\C$ to $Z_X$.
\end{theorem}

\begin{proof}
Elementary estimates using the finiteness of the graph show that the sum $\sum_ce^{-sl(c)}$ converges for $\Re(s)$ large enough. This settles convergence.
The  argument for the second assertion goes back to Hyman Bass \cite{Bass}. We only provide the necessary changes for the proof to work for metric graphs.
Let $\tilde E$ be the set of all oriented edges, i.e., $\tilde E$ is the pullback of $E\to \CP(V)$ under the natural map $V\times V\to\CP(V)$; $(a,b)\mapsto\{a,b\}$.
More precisely, the elements of $\tilde E$ are all pairs $(k,(a,b))$ in $E\times (V\times V)$ with the property that $\ph(k)=\{a,b\}$.
For a given oriented edge $(k,(a,b))$ the \e{opposite edge} is defined by
$\ol{(k,(a,b))}=(k,(b,a))$.

For $(k,(a,b))\in\tilde E$ we say that $a$ is the \e{source} and $b$ the \e{target} of $(k,(a,b))$.
Let $F(\tilde E)$ be the free $\C$-vector space spanned by $\tilde E$ and for $s\in\C$ define the linear operator $T_s:F(\tilde E)\to F(\tilde E)$ by
$$
T_s(k)=\sum_{\stack{k'\ne \bar k}{s(k')=t(k)}}e^{-sl(k')}k',
$$
where $k\in\tilde E$ now and $\bar k$ is the edge opposite to $k$.
We equip $F(\tilde E)$ with the inner product making the given basis $(k)_{k\in\tilde E}$ an orthonormal basis.
We denote by $L$ the length function which assigns the length $1$ to every edge.
Then for any $n\in\N$, the trace of $T^n$ is
\begin{align*}
\tr(T^n)&= \sum_{k}\sp{T^nk,k}\\
&= \sum_{c:L(c)=n}e^{-sl(c)}L(c_0),
\end{align*}
where the sum runs over all closed geodesics and $c_0$ is the unique primitive geodesic underlying a given closed geodesic $c$.
The function $s\mapsto\det(1-T_s)$ is entire and is a polynomial in $e^{-sk}$, where $k$ ranges in $E$.
We will show that this function equals $Z_X(s)$ from which all assertions of the proposition follow.
We keep the notation marking primitive geodesics by $c_0$ and compute
\begin{align*}
Z(s)^{-1}&=\prod_{c_0}(1-e^{-sl(c_0)})=\exp\(\sum_{c_0}\log(1-e^{-sl(c_0)})\)\\
&=\exp\(-\sum_{c_0}\sum_{n=1}^\infty\frac{e^{-snl(c_0)}}n\)\\
&=\exp\(-\sum_{c}\frac{e^{-sl(c)}}{L(c)}L(c_0)\)= \exp\(-\sum_{n=1}^\infty\frac1n\tr(T_s^n)\)\\
&=\exp\(\tr\log\(1-T_s\)\)=\det(1-T_s)\tag*\qedhere
\end{align*}
\end{proof}

\begin{definition}
We call a metric graph $X$ a \e{rational graph}, if there exists $\theta>0$ such that all lengths of closed geodesics lie in $\theta\Q$.
This is for instance the case, if all lengths of edges are in $\theta\Q$.
\end{definition}

\begin{definition}
A connected graph $X$ is called \e{non-tivial} if it cannot be contracted to a point orr a circle.
This is equivalent to $X$ having infinitely many primitive closed geodesics.
\end{definition}

For a metric graph $X$, we write $N(X)$ for the sequence of the values $N(c)$ in ascending order and without repetition and we write $\LG(X)=\LG(N(X))$.

\begin{theorem}
For a  non-trivial finite graph $X$ one has 
$$
\LG(X)=\begin{cases}\infty& X\text{ is rational},\\ 0&\text{otherwise}.\end{cases}
$$
\end{theorem}

\begin{proof}
Assume first that $X$ is rational.
As $X$ is finite, we can assume that there is $\theta>0$ such that all lengths of edges are in $\theta\N$.
So the sequence $N(X)$ is a subsequence of the sequence $(e^{\theta n})_{n\in\N}$, therefore
$$
\LG(X)\ge\LG\((e^{\theta n})\).
$$
As $e^{\theta n}=\(e^{\theta}\)^n$, the right hand side is $+\infty$ by Example \ref{Ex1.1}.

Now assume that $X$ is not rational.
After scaling, the graph contains two closed geodesics $c_a$ and $c_1$ of lengths $a>0$ and $1$ resp., with $a\notin\Q$. For a pair of natural numbers $(m,n)\in\N^2$ we can iterate the geodesic $c_a$ for $m$ times, then iterate $c_1$ for $n$ times and then go back to the starting point to get a closed geodesic of total length $l(m,n)=am+n+2\delta$, where $\delta>0$ is the distance between a given vertex in $c_a$ and one in $c_b$.
For another pair $(k,l)$ we have
\begin{align*}
e^{l(m,n)}-e^{l(k,l)}&=
e^{ma+n+2\delta}-e^{ka+l+2\delta}\\
&=\(e^{a(m-k)+(n-l)}-1\)e^{al+l+2\delta}.
\end{align*}
For this to remain bounded for infinitely many geodesics, the bracket expression must be small, so the exponent $a(m-k)+(n-l)$ must be close to zero.
But close to zero, $e^x-1$ behaves like $x$, so that we consider
$$
\left|a(m-k)+(n-l)\right|e^{ak+l+2\delta}.
$$
For $a>0$, this expression is less than a given $C>0$ if
\begin{align*}
\left|a-\frac{n-l}{m-k}\right|\le \frac{C}{e^{ak+l+2\delta}(m-k)}.
\end{align*}
Fix $k$ and $l$ and let $m$ and $n$ vary. Setting $p=n-l$ and $q=m-k$ we have to show that there are infinitely many pairs $(p,q)\in\N^2$ such that
$$
\left| a-\frac pq\right|\le \frac D q
$$
for any given $D>0$. This is well known, see \cite{Chandra}.
\end{proof}

\section{Riemann surfaces}
Let $X$ denote a connected Riemann surface equipped with a hyperbolic metric such that $X$ has finite volume.
Then the universal covering of $X$ is the upper half plane $\H$ in $\C$ and so $X$ is the quotient $\Ga\bs \H$, where $\Ga$ is the fundamental group of $X$, acting as biholomorphic maps on $\H$, hence $\Ga$ may be viewed as a discrete subgroup of $\PSL_2(\R)=\SL_2(\R)/\pm 1$, which through the action by linear fractionals,
$
\smat abcd z=\frac{az+b}{cz+d}
$
is identified with the group of biholomorphic maps on $\H$.

The Ruelle zeta function \cite{Ruelle, BO} is defined as
$$
R(s)=\prod_{c}(1-N(c)^{-s})^{-1},
$$ 
where the product runs over all primitive closed geodesics on $X$ and $N(c)=e^{l(c)}$, where $l(c)$ is the length of the closed geodesic $c$.
Here a closed geodesic is called \e{primitive}, if it is not a power of a shorter one. 
Using the identification between closed geodesics and hyperbolic conjugacy classes in $\Ga$, one extends this theory to include non torsion-free lattices $\Ga$ like $\PSL_2(\Z)$.
It is known, \cites{Ruelle, BO,HA2}, that $R(s)$ extends to a meromorphic function on $\C$ and that in the compact case all its zeros and poles lie in the union of $\R$ with the two vertical lines $\{\Re(s)=\pm \frac12\}$.
If one compares the Ruelle function with the Riemann zeta function, the role of the primes is taken over by the numbers $N(c)$.
So let $N(\Ga)$ denote the sequence of values $N(c)$ in ascending order.
Note that multiplicities occur, i.e., there may be different closed geodesics $c,c'$ which are not equal and not inverse to each other, such that $l(c)=l(c')$.
We plainly ignore multiplicities, as they don't influence our question. 
We shall write $\LG(\Ga)$ for $\LG(N(\Ga))$.
We call $\Ga$ a \e{bounded gaps group} if the sequence $N(\Ga)$ has bounded gaps, i.e., if $0<\LG(\Ga)<\infty$ holds.

If $\Sigma$ is a finite index subgroup of $\Ga$, then one has
$$
\LG(\Sigma)\ge \LG(\Ga).
$$
But the following question remains open.

\begin{question}\label{quest}
If $\Sigma$ is a finite index subgroup of $\Ga$, is it true that 
\begin{itemize}
\item $\LG(\Ga)<\infty\quad\Rightarrow\quad \LG(\Sigma)<\infty$?
\item $\LG(\Sigma)>0\quad\Rightarrow\quad \LG(\Ga)>0$?
\end{itemize}

\end{question}

In the next section we shall prove that principal congruence subgroups of $\PSL_2(\Z)$ and unit groups of quaternion algebras have bounded gaps.
Since any arithmetic group is commensurable to one of these, a positive answer to the above question would imply that every arithmetic groups has bounded gaps. 

\begin{theorem}
 One has $\LG(\Ga)=0$ for $\Ga$ in a dense subset of the Teichmüller space.
\end{theorem}

\begin{proof}
Let $g\ge 2$ be a natural number.
Let $\CT_g$ denote the Teichmüller space of all compact surfaces of genus $g$.
Let $X\in\CT_g$ and for $x>0$ let $v(x)$ be the number of closed geodesics $c$ with $N(c)\le x$.
By the Prime Geodesic Theorem \cite{Huber,Hejhal,Randol} one has
$$
v(x)\sim \frac{x^2}{\log(x^2)}
$$
as $x\to\infty$.
By Remark \ref{rem1.1}, this implies $\LG(\Ga)=0$ if there are no multiplicities, i.e., if  different closed geodesics don't share the same length.
The set of $X\in\CT_g$ without multiplicities is a dense set, as is shown in \cite{McShane}.
\end{proof}

\section{Arithmetic groups}
For $N\in\N$, $N\ge 2$, we consider the principal congruence subgroup $\Ga(N)$, which is the kernel of the group homomorphism $\PSL_2(\Z)\to\PSL_2(\Z/N\Z)$.
Recall that s subgroup $\Sigma\subset\Ga$ is called a \e{congruence subgroup}, if $\Sigma$ contains $\Ga(N)$ for some $N\in\N$.

\begin{theorem}\label{prop1.2}
\begin{enumerate}[\rm (a)]
\item Let $\Ga=\PSL_2(\Z)$. Then $\LG(\Ga)=1$, so $\Ga$ has bounded gaps.
\item Every congruence subgroup of $\Ga$ has bounded gaps.
\end{enumerate}
\end{theorem}

\begin{proof}
(a) Let $\bar\ga\in\Ga$ be a hyperbolic element and let $\ga\in\SL_2(\Z)$ be a preimage of $\ga$.
Let $\M_2(\Q)$ be the algebra of $2\times 2$ matrices and let $\M_2(\Q)_\ga$ be the centralizer of $\ga$.
This then is a real quadratic field and $\CO_\ga=\M_2(\Q)_\ga\cap\M_2(\Z)$ is an order in this field, so it is isomorphic to some
$$
\CO_D=\left\{\frac{a+b\sqrt D}2:a\equiv bD\mod(2)\right\},
$$
where $D\in\N$, $D\equiv 0,1\mod(4)$ not a square.
In this way one gets a map from the set of hyperbolic conjugacy classes in $\Ga$ to the set of isomorphism classes of orders in real quadratic fields, which are parametrized by $D$ as above, under which a given order $\CO_D$ has exactly $h(D)$ preimages, where $h(D)$ is the class number (see \cite{class}).

The unit group $\CO_D^\times$ consist of all $\al\in\CO_D$ of norm $N(\al)=\al\ol\al=\pm 1$, where we have written $\ol{a+b\sqrt D}=a-\sqrt D$ for the non-trivial Galois homomorphism.
The subgroup $\CO_D^1$ of all $\al$ with $N(\al)=1$ has index one or two.
This group is isomorphic to $\{\pm 1\}\times \eps_D^\Z$ for a unique \e{fundamental unit} $\eps_D>1$.

If $\ga$ is primitive, it gets mapped to $\pm\eps_D$ and we choose $\ga$ in a way that it gets mapped to $\eps_D$.
We then find that actually,
$$
\eps_D=N(\ga)=e^{l(\ga)}.
$$
We say that $D\equiv 0,1\mod(4)$ is a \e{fundamental discriminant}, if $D$ is square-free or $D/4$ is square-free respectively.
A variation of the arguments  in \cite{Sprind} shows that
the number of $\eps_D\le x$ for $D$ being a fundamental unit, is $x+O(x^{1/2})$.
Hence the number of $\eps_D\le x$ grows at least like $x$.
In the expression $\eps_D=\frac{a+b\sqrt D}2$ we can replace $b$ by $1$ and $D$ by $Db^2$ without changing the unit $\eps_D$. Because of $a^2-b\sqrt D=4$ we get $\eps_D=\frac{a+\sqrt {a^2-4}}2$.
In particular, $\eps_D$ is determined by $a$ and it satisfies
$$
a-\eps_D(a)=\frac 2{a+\sqrt{a^2-4}},
$$
and the latter tends to zero as $a$ grows.
This implies that $\eps_D$ is close to an integer for large $D$, which implies that $\LG(\Ga)\ge 1$.
Assume $\LG(\Ga)>1$, then by the same argument we infer $\LG(\Ga)\ge 2$ and deduce
$$
\#\{\eps_D\le x\}=\le\frac x2+O(x^{1/2}),
$$
which contradicts the results of  \cite{Sprind}.
So part (a) follows.

Now for (b). Let $N\in\N$.
We have already seen that $\eps_D=\frac{a+b\sqrt D}2$ is determined by $a$.
As the number of $\eps_D\le x$ grows like $x$, it follows that the set of all $a\in\N$, such that $\frac{a+\sqrt{a^2-4}}2$ is a unit $\eps_D$, has density one.
Therefore the set of all such $a$ satisfying  $a\equiv 2\mod(N^4)$ has positive density $\frac1{N^4}$.
So, if we show that for every such $a$, there is an element of $\Ga(N)$, which as element of $\Ga(1)$ corresponds to $\eps_D=\frac{a+\sqrt{a^2-4}}2$, we have finished the proof.

So let $a$ satisfy $a\equiv 2\mod (N^4)$.
Write $a^2-4=b^2D$, where $D$ is a fundamental discriminant, so $b^2$ contains all square factors away from $2$.
By the choice of $a$, the number $N$ divides $b$.
We show that the element
$$
\ga=\mat{\frac{a+bD}2}{b\frac{D(D-1)}4}{b}{\frac{a-bD}2}
$$
does the job.
Since $a\equiv bD\mod 2$, the diagonal entries are integers.
Since either $D$ or $D-1$ is divisible by $4$, the off-diagonal entries are integers as well.
The trace is $a$ and the determinant is one, so $\ga$ corresponds to $\eps_D$ indeed.
It remains to show that $\ga$ lies in $\Ga(N)$.
As $N$ divides $b$, the off-diagonal entries are divisible by $N$, so $\ga$ is congruent to the diagonal matrix $\diag\(\frac{a+bD}2,\frac{a-bD}2\)$ modulo $N$.
Write $a=2+cN^4=2+lN$, so $l=cN^4$ and write $b=kN$, then
$$
\frac{a+bD}2=\frac{2+(l+kD)N}{2}=1+\frac{l+kD}2N.
$$
If $N$ is odd, $l+kD$ is even and $\frac{a+bD}2\equiv 1\equiv \frac{a-bD}2\mod(N)$.
If $N$ is even, so is $l$ and one has $k^2N^2D=b^2D=a^2-4=4cN^4+c^2N^8$, which implies that $k$ is even and so $\ga\in\Ga(N)$.
\end{proof}

Let $M$ be a non-split quaternion algebra over $\Q$ which splits over $\R$.
This means that there is an isomorphism $\phi:M(\R)\to M_2(\R)$ with $N_\red(\al)=\det(\phi(\al))$ for every $\al\in M(\R)$, where $N_\red$ is the reduced norm.
For any ring $R$, let $M^1(R)$ denote the group of all $\al\in M(R)$ with $N_\red(\al)=1$. Then $\phi$ induces an isomorphism $M^1(\R)\cong \SL_2(\R)$.

\begin{theorem}\label{prop1.3}
The arithmetic group $\Ga=\phi(M^1(\Z))/\pm 1\subset\PSL_2(\R)$ has bounded gaps.
\end{theorem}

\begin{proof}
Let $S$ be the finite the set of primes $p$ at which $M$ does not split and let $R=\prod_{p\in S}p$ be its radical. Note that $S$ has at least two elements.
For a given $c\in\N$ which is coprime to $2R$ we set $a=2+4cR$ and $D=a^2-4=16cR(1+cR)$.
Write $D=m^2\Delta$, where $\Delta$ is square-free.
Then the discriminant $\disc(K)$ of $K=\Q(\sqrt D)$ is equal to $\Delta$ if $\Delta\equiv 1\mod 4$ and $4\Delta$ otherwise.
We now argue that every $p\in S$ is ramified in $K$, i.e. $p|\disc(K)$. If $p$ is odd, then $p$ divides $D$, but $p^2$ does not, hence $p|\Delta$, so $p$ divides the discriminant.
If $2\in S$, then from $D=16cR(1+cR)$ we see that $2|\Delta$, but $4\nmid\Delta$, so $\Delta\equiv 2\mod 4$, which means that $\disc(K)=4\Delta$ is even.
This implies that all primes in $S$ are ramified in $K$.
Next we need to show that $\CO_D$ is maximal at each $p\in S$.
This, however, follows from $p|\disc(K)$ and $p^2\nmid\disc(K)$ which we have shown for each $p\in S$.
We have shown that each $p\in S$ is non-decomposed in $K$ and that $\CO_D$ is maximal at $p$.
This implies (see \cite{class}), that there exists a primitive $\ga\in\Ga$ such that $N(\ga)=\eps_D$.
But the set of all $a=2+4cR$ contains infinite arithmetic progressions, which implies $\LG(\Ga)<\infty$.
Since, on the other hand, the sequence $N(\Ga)$ is a subsequence of $N(\PSL_2(\Z))$, we also have $\LG(\Ga)>0$.
\end{proof}

{\small Mathematisches Institut\\
Auf der Morgenstelle 10\\
72076 T\"ubingen\\
Germany\\
\tt deitmar@uni-tuebingen.de}

\today

\end{document}